\documentclass[11pt]{article}
\usepackage{amssymb,amsfonts,amsmath,latexsym,epsf,tikz,url}
\usepackage{graphics}
\usepackage[usenames,dvipsnames]{pstricks}
\usepackage{pstricks-add}
\usepackage{epsfig}
\usepackage{pst-grad} 
\usepackage{pst-plot} 
\usepackage[space]{grffile} 
\usepackage{etoolbox} 
\makeatletter 
\patchcmd\Gread@eps{\@inputcheck#1 }{\@inputcheck"#1"\relax}{}{}
\makeatother

\newtheorem{theorem}{Theorem}[section]
\newtheorem{proposition}[theorem]{Proposition}

\newtheorem{corollary}[theorem]{Corollary}

\newtheorem{remark}[theorem]{Remark}
\newtheorem{example}[theorem]{Example}

\newcommand{\qed}{\hfill $\square$\medskip}

\begin{document} 

\title{Co-even Domination Number of a Modified Graph by Operations 
	on a Vertex or an Edge}

\author{
Nima Ghanbari$^1$ \and Saeid Alikhani$^1$ \and Mohammad Ali Dehghanizadeh$^2$
}

\date{\today}

\maketitle

\begin{center}
$^1$Department of Mathematical Sciences, Yazd University, 89195-741, Yazd, Iran\\
\medskip
$^2$Department of Mathematics, Technical and Vocational University, Tehran, Iran
\bigskip

{\tt n.ghanbari.math@gmail.com, ~~alikhani@yazd.ac.ir,~~Mdehghanizadeh@tvu.ac.ir }
\end{center}


\begin{abstract}
Let $G=(V,E)$ be a simple graph. A dominating set of $G$ is a subset $D\subseteq V$ such that every vertex not in $D$ is adjacent to at least one vertex in $D$.
The cardinality of a smallest dominating set of $G$, denoted by $\gamma(G)$, is the domination number of $G$. A dominating set $D$ is called co-even dominating set if the degree of vertex $v$ is even number for all $v\in V-D$. 
 The cardinality of a smallest co-even dominating set of $G$, denoted by $\gamma _{coe}(G)$, is the co-even domination number of $G$.
In this paper we study co-even domination number of graphs which constructed by some operations on a vertex or an edge of a graph.
\end{abstract}


\medskip
\noindent{\bf AMS Subj.\ Class.:}  05C69, 05C76

\section{Introduction}

Let $G = (V,E)$ be a  graph with vertex set $V$ and edge set $E$. Throughout this paper, we consider graphs without loops and directed edges.
For each vertex $v\in V$, the set $N_G(v)=\{u\in V | uv \in E\}$ refers to the open neighbourhood of $v$ and the set $N_G[v]=N_G(v)\cup \{v\}$ refers to the closed neighbourhood of $v$ in $G$. The degree of $v$, denoted by $\deg(v)$, is the cardinality of $N_G(v)$.
 A set $D\subseteq V$ is a  dominating set if every vertex in $V- D$ is adjacent to at least one vertex in $D$.
The  domination number $\gamma(G)$ is the minimum cardinality of a dominating set in $G$. There are various domination numbers in the literature.
For a detailed treatment of domination theory, the reader is referred to \cite{domination}.

\medskip

A dominating set $D$ is called a co-even dominating set, if the degree of 
each vertex in 
$ V-D$ is even (\cite{Sha}). The cardinality of a smallest co-even dominating set of $G$, denoted by $\gamma _{coe}(G)$, is the co-even domination number of $G$.  Ghanbari in \cite{Nima}, considered  binary operations of graphs and presented some bounds  for co-even domination number of
 join, corona, neighbourhood corona, and Haj\'{o}s sum of two graphs. In this paper we present more results for co-even domination number.
 
\medskip

In the next Section, first we mention the definition of vertex and edge removal of a graph and then study the co-even domination number  of a graph which constructed by vertex and edge removal and find some bounds for them. In Section 3, we mention the definition of vertex and edge contraction and  study co-even domination number of vertex and edge contraction of a graph. We find some bounds for them, and finally, we present upper and lower bounds for co-even domination number of a graph regarding vertex (edge) removal and contraction.

\section{Vertex and edge removal of a graph}
The graph $G-v$ is a graph that is made  by deleting the vertex $v$ and all edges connected to $v$ from the graph $G$ and the graph $G-e$ is a graph that obtained from $G$ by simply removing the  edge $e$. Our main results in this section are  in obtaining 
some bounds for the co-even domination number of vertex and edge removal of a graph. First we state some known results.

	\begin{proposition}{\rm\cite{Sha}}\label{pro-sha}
Let $G=(V,E)$ be a graph and $D$ is a co-even dominating set of $G$. Then,
\begin{itemize}
\item[(i)]
All vertices of odd or zero degrees belong to every co-even dominating set.
\item[(ii)]
$\deg(v)\geq 2$, for all $v\in V-D$.
\item[(iii)]
$\gamma (G) \leq \gamma_{coe} (G).$
\end{itemize}
	\end{proposition}

By the definition of co-even domination number, we have the following easy result:

\begin{proposition}\label{pro-disconnect}
	Let $G$ be a disconnected graph with components $G_1$ and $G_2$. Then
	$$\gamma _{coe}(G)=\gamma _{coe}(G_1)+\gamma _{coe}(G_2).$$
\end{proposition}

Now we consider to the vertex removal of a graph and find an upper bound and a lower one for co-even domination number of the constructed graph.

\begin{theorem}\label{G-v}
Let $G=(V,E)$ be a graph and $v\in V$. Then,
$$\gamma _{coe}(G) - \deg(v) -1 \leq \gamma _{coe}(G-v)\leq \gamma _{coe}(G) + \deg(v) -1.$$
\end{theorem}

\begin{proof}
 Suppose that $v\in V$ and $D_{coe}(G)$ is co-even dominating set of $G$. First we find the upper bound for $\gamma _{coe}(G-v)$. We consider the following cases:
\begin{itemize}
\item[(i)]
$\deg(v)$ is even and $v\notin D_{coe}(G)$. Then at least one of the neighbours of $v$ should be in $D_{coe}(G)$. Now by adding all other neighbours of $v$ in $D_{coe}(G)$, we have a co-even dominating set for $G-v$. The size of this set is at most
$\gamma _{coe}(G) -1 + \deg(v) $.
\item[(ii)]
$\deg(v)$ is even and $v\in D_{coe}(G)$. By removing $v$ and all edges related to that, some of the vertices in its neighbour, may not dominate with any other vertex. So by adding all of the neighbours of $v$ in $D_{coe}(G)-\{v\}$, we have a co-even dominating set for $G-v$. So $\gamma _{coe}(G-v)\leq \gamma _{coe}(G) + \deg(v) -1$.
\item[(iii)]
$\deg(v)$ is odd. Then by Proposition \ref{pro-sha}, $v\in D_{coe}(G)$. Now by the same argument as case (ii), we have 
$\gamma _{coe}(G-v)\leq \gamma _{coe}(G) + \deg(v) -1$.
\end{itemize}
Therefore $\gamma _{coe}(G-v)\leq \gamma _{coe}(G) + \deg(v) -1$. Now we find the lower bound for $\gamma _{coe}(G-v)$.
First we remove $v$ and all corresponding edges to that. Now we find a co-even dominating set for $G-v$. Suppose that this set is $D_{coe}(G-v)$. One can easily check that $D_{coe}(G-v)\cup N_G[v]$ is a co-even dominating set of $G$. So 
$$\gamma _{coe}(G)\leq \gamma _{coe}(G-v) + \deg(v)+1,$$
and therefore we have the result.\qed
\end{proof}

\begin{remark}
The  bounds in Theorem \ref{G-v} are sharp. For the upper bound, it suffices to consider $G$ as shown in Figure \ref{fig1}. The set of black vertices in $G$ is a co-even dominating set of $G$. Now, by removing vertex $v$ with degree $4$, the set of black vertices is a co-even dominating set of $G-v$, and $\gamma _{coe}(G-v)= \gamma _{coe}(G) + \deg(v) -1$. For the lower bound, it suffices to consider $H$ as shown in Figure \ref{h-vlower}. The set of black vertices in $H$ and $H-v$ are co-even dominating sets of them, respectively. So $\gamma _{coe}(H-v)= \gamma _{coe}(H) - \deg(v) -1$.
\end{remark}

	\begin{figure}
		\begin{center}
			\psscalebox{0.5 0.5}
{
\begin{pspicture}(0,-8.299306)(14.5971155,-0.097916566)
\psline[linecolor=black, linewidth=0.08](0.4,-2.2993054)(2.4,-0.2993054)(2.4,-0.2993054)
\psline[linecolor=black, linewidth=0.08](0.4,-2.2993054)(2.4,-2.2993054)(2.4,-2.2993054)
\psline[linecolor=black, linewidth=0.08](0.4,-2.2993054)(2.4,-4.2993054)(2.4,-4.2993054)
\psline[linecolor=black, linewidth=0.08](0.4,-2.2993054)(2.4,-6.2993054)(2.4,-6.2993054)
\psline[linecolor=black, linewidth=0.08](2.4,-2.2993054)(3.2,-1.4993054)(3.2,-3.0993054)(2.4,-2.2993054)(2.4,-2.2993054)
\psline[linecolor=black, linewidth=0.08](2.4,-0.2993054)(6.0,-2.2993054)(6.0,-2.2993054)
\psline[linecolor=black, linewidth=0.08](2.4,-4.2993054)(6.0,-2.2993054)(6.0,-2.2993054)
\psline[linecolor=black, linewidth=0.08](2.4,-6.2993054)(6.0,-2.2993054)(6.0,-2.2993054)
\psdots[linecolor=black, dotsize=0.4](2.4,-2.2993054)
\psdots[linecolor=black, dotsize=0.4](6.0,-2.2993054)
\psline[linecolor=black, linewidth=0.08](10.8,-2.2993054)(11.6,-1.4993054)(11.6,-3.0993054)(10.8,-2.2993054)(10.8,-2.2993054)
\psline[linecolor=black, linewidth=0.08](10.8,-0.2993054)(14.4,-2.2993054)(14.4,-2.2993054)
\psline[linecolor=black, linewidth=0.08](10.8,-4.2993054)(14.4,-2.2993054)(14.4,-2.2993054)
\psline[linecolor=black, linewidth=0.08](10.8,-6.2993054)(14.4,-2.2993054)(14.4,-2.2993054)
\psdots[linecolor=black, dotsize=0.4](10.8,-2.2993054)
\psdots[linecolor=black, dotsize=0.4](14.4,-2.2993054)
\psdots[linecolor=black, dotsize=0.4](10.8,-0.2993054)
\psdots[linecolor=black, dotsize=0.4](10.8,-4.2993054)
\psdots[linecolor=black, dotsize=0.4](10.8,-6.2993054)
\psdots[linecolor=black, dotstyle=o, dotsize=0.4, fillcolor=white](11.6,-1.4993054)
\psdots[linecolor=black, dotstyle=o, dotsize=0.4, fillcolor=white](11.6,-3.0993054)
\psdots[linecolor=black, dotstyle=o, dotsize=0.4, fillcolor=white](3.2,-1.4993054)
\psdots[linecolor=black, dotstyle=o, dotsize=0.4, fillcolor=white](3.2,-3.0993054)
\psdots[linecolor=black, dotstyle=o, dotsize=0.4, fillcolor=white](2.4,-0.2993054)
\psdots[linecolor=black, dotstyle=o, dotsize=0.4, fillcolor=white](0.4,-2.2993054)
\psdots[linecolor=black, dotstyle=o, dotsize=0.4, fillcolor=white](2.4,-4.2993054)
\psdots[linecolor=black, dotstyle=o, dotsize=0.4, fillcolor=white](2.4,-6.2993054)
\rput[bl](2.0,-8.299305){\LARGE{$G$}}
\rput[bl](11.6,-8.299305){\LARGE{$G-v$}}
\rput[bl](0.0,-2.6993055){$v$}
\end{pspicture}
}
		\end{center}
		\caption{Graphs $G$ and $G-v$} \label{fig1}
	\end{figure}
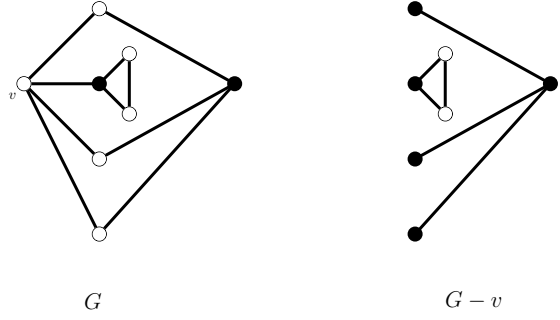

	\begin{figure}
		\begin{center}
			\psscalebox{0.5 0.5}
{
\begin{pspicture}(0,-5.1014423)(15.801389,1.4956731)
\psdots[linecolor=black, dotsize=0.4](4.0,1.2985578)
\psdots[linecolor=black, dotsize=0.4](4.0,0.49855775)
\psline[linecolor=black, linewidth=0.08](4.0,1.2985578)(2.4,0.8985577)(4.0,0.49855775)(4.0,0.49855775)
\psdots[linecolor=black, dotsize=0.4](4.0,-0.70144224)
\psdots[linecolor=black, dotsize=0.4](4.0,-1.5014423)
\psline[linecolor=black, linewidth=0.08](4.0,-0.70144224)(2.4,-1.1014422)(4.0,-1.5014423)(4.0,-1.5014423)
\psdots[linecolor=black, dotsize=0.4](4.0,-2.7014422)
\psdots[linecolor=black, dotsize=0.4](4.0,-3.5014422)
\psline[linecolor=black, linewidth=0.08](4.0,-2.7014422)(2.4,-3.1014423)(4.0,-3.5014422)(4.0,-3.5014422)
\psline[linecolor=black, linewidth=0.08](2.4,0.8985577)(0.4,-1.1014422)(2.4,-1.1014422)(0.4,-1.1014422)
\psline[linecolor=black, linewidth=0.08](0.4,-1.1014422)(2.4,-3.1014423)(2.4,-3.1014423)
\psline[linecolor=black, linewidth=0.08](4.0,1.2985578)(6.8,-0.30144227)(6.8,-0.30144227)
\psline[linecolor=black, linewidth=0.08](4.0,-3.5014422)(6.8,-1.9014423)(6.8,-1.9014423)
\psline[linecolor=black, linewidth=0.08](4.0,0.49855775)(6.8,-0.30144227)(6.8,-0.30144227)
\psline[linecolor=black, linewidth=0.08](4.0,-0.70144224)(6.8,-0.30144227)(6.8,-0.30144227)
\psline[linecolor=black, linewidth=0.08](4.0,-1.5014423)(6.8,-0.30144227)(6.8,-0.30144227)
\psline[linecolor=black, linewidth=0.08](4.0,-2.7014422)(6.8,-0.30144227)(6.8,-0.30144227)
\psline[linecolor=black, linewidth=0.08](4.0,-3.5014422)(6.8,-0.30144227)(6.8,-0.30144227)
\psline[linecolor=black, linewidth=0.08](4.0,-2.7014422)(6.8,-1.9014423)
\psline[linecolor=black, linewidth=0.08](4.0,-1.5014423)(6.8,-1.9014423)
\psline[linecolor=black, linewidth=0.08](4.0,-0.70144224)(6.8,-1.9014423)
\psline[linecolor=black, linewidth=0.08](4.0,0.49855775)(6.8,-1.9014423)(6.8,-1.9014423)
\psline[linecolor=black, linewidth=0.08](4.0,1.2985578)(6.8,-1.9014423)
\psdots[linecolor=black, dotstyle=o, dotsize=0.4, fillcolor=white](6.8,-0.30144227)
\psdots[linecolor=black, dotstyle=o, dotsize=0.4, fillcolor=white](6.8,-1.9014423)
\psdots[linecolor=black, dotsize=0.4](12.8,1.2985578)
\psdots[linecolor=black, dotsize=0.4](12.8,0.49855775)
\psline[linecolor=black, linewidth=0.08](12.8,1.2985578)(11.2,0.8985577)(12.8,0.49855775)(12.8,0.49855775)
\psdots[linecolor=black, dotsize=0.4](12.8,-0.70144224)
\psdots[linecolor=black, dotsize=0.4](12.8,-1.5014423)
\psline[linecolor=black, linewidth=0.08](12.8,-0.70144224)(11.2,-1.1014422)(12.8,-1.5014423)(12.8,-1.5014423)
\psdots[linecolor=black, dotsize=0.4](12.8,-2.7014422)
\psdots[linecolor=black, dotsize=0.4](12.8,-3.5014422)
\psline[linecolor=black, linewidth=0.08](12.8,-2.7014422)(11.2,-3.1014423)(12.8,-3.5014422)(12.8,-3.5014422)
\psline[linecolor=black, linewidth=0.08](12.8,1.2985578)(15.6,-0.30144227)(15.6,-0.30144227)
\psline[linecolor=black, linewidth=0.08](12.8,-3.5014422)(15.6,-1.9014423)(15.6,-1.9014423)
\psline[linecolor=black, linewidth=0.08](12.8,0.49855775)(15.6,-0.30144227)(15.6,-0.30144227)
\psline[linecolor=black, linewidth=0.08](12.8,-0.70144224)(15.6,-0.30144227)(15.6,-0.30144227)
\psline[linecolor=black, linewidth=0.08](12.8,-1.5014423)(15.6,-0.30144227)(15.6,-0.30144227)
\psline[linecolor=black, linewidth=0.08](12.8,-2.7014422)(15.6,-0.30144227)(15.6,-0.30144227)
\psline[linecolor=black, linewidth=0.08](12.8,-3.5014422)(15.6,-0.30144227)(15.6,-0.30144227)
\psline[linecolor=black, linewidth=0.08](12.8,-2.7014422)(15.6,-1.9014423)
\psline[linecolor=black, linewidth=0.08](12.8,-1.5014423)(15.6,-1.9014423)
\psline[linecolor=black, linewidth=0.08](12.8,-0.70144224)(15.6,-1.9014423)
\psline[linecolor=black, linewidth=0.08](12.8,0.49855775)(15.6,-1.9014423)(15.6,-1.9014423)
\psline[linecolor=black, linewidth=0.08](12.8,1.2985578)(15.6,-1.9014423)
\psdots[linecolor=black, dotstyle=o, dotsize=0.4, fillcolor=white](15.6,-0.30144227)
\psdots[linecolor=black, dotstyle=o, dotsize=0.4, fillcolor=white](15.6,-1.9014423)
\psdots[linecolor=black, dotstyle=o, dotsize=0.4, fillcolor=white](11.2,0.8985577)
\psdots[linecolor=black, dotstyle=o, dotsize=0.4, fillcolor=white](11.2,-1.1014422)
\psdots[linecolor=black, dotstyle=o, dotsize=0.4, fillcolor=white](11.2,-3.1014423)
\psdots[linecolor=black, dotsize=0.4](2.4,0.8985577)
\psdots[linecolor=black, dotsize=0.4](2.4,-1.1014422)
\psdots[linecolor=black, dotsize=0.4](2.4,-3.1014423)
\psdots[linecolor=black, dotsize=0.4](0.4,-1.1014422)
\rput[bl](2.8,-5.1014423){\LARGE{$H$}}
\rput[bl](12.0,-5.1014423){\LARGE{$H-v$}}
\rput[bl](0.0,-1.5014423){$v$}
\end{pspicture}
}
		\end{center}
		\caption{Graphs $H$ and $H-v$} \label{h-vlower}
	\end{figure}
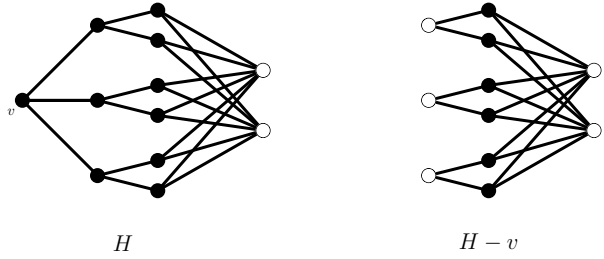

The following example shows that $\gamma _{coe}(G)$ and $\gamma _{coe}(G-v)$ can be equal.

\begin{example}
Consider the graphs $H$ and $H-v$ as shown in Figure \ref{fig2-example}. Obviously, the set of black vertices of each graph is co-even dominating set with smallest size. Therefore, there are some graphs $G$ such that $\gamma _{coe}(G)=\gamma _{coe}(G-v)$.
\end{example}

	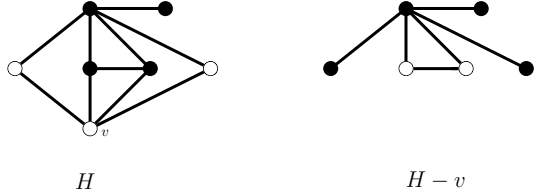
\begin{figure}[!h]
		\begin{center}
			\psscalebox{0.5 0.5}
{
\begin{pspicture}(0,-5.5014424)(13.998505,-0.50432676)
\psline[linecolor=black, linewidth=0.08](0.20138885,-2.3014421)(2.2013888,-0.70144224)(2.2013888,-2.3014421)(2.2013888,-3.9014423)(0.20138885,-2.3014421)(0.20138885,-2.3014421)
\psline[linecolor=black, linewidth=0.08](2.2013888,-0.70144224)(3.8013887,-2.3014421)(2.2013888,-3.9014423)(2.2013888,-3.9014423)
\psline[linecolor=black, linewidth=0.08](2.2013888,-3.9014423)(5.4013886,-2.3014421)(2.2013888,-0.70144224)(2.2013888,-0.70144224)
\psline[linecolor=black, linewidth=0.08](2.2013888,-0.70144224)(4.201389,-0.70144224)(4.201389,-0.70144224)
\psline[linecolor=black, linewidth=0.08](10.601389,-0.70144224)(12.601389,-0.70144224)(12.601389,-0.70144224)
\psdots[linecolor=black, dotsize=0.4](3.8013887,-2.3014421)
\psdots[linecolor=black, dotsize=0.4](2.2013888,-2.3014421)
\psdots[linecolor=black, dotsize=0.4](2.2013888,-0.70144224)
\psdots[linecolor=black, dotsize=0.4](4.201389,-0.70144224)
\psdots[linecolor=black, fillstyle=solid, dotstyle=o, dotsize=0.4, fillcolor=white](0.20138885,-2.3014421)
\psdots[linecolor=black, fillstyle=solid, dotstyle=o, dotsize=0.4, fillcolor=white](5.4013886,-2.3014421)
\psdots[linecolor=black, fillstyle=solid, dotstyle=o, dotsize=0.4, fillcolor=white](2.2013888,-3.9014423)
\psline[linecolor=black, linewidth=0.08](2.2013888,-2.3014421)(3.8013887,-2.3014421)(3.8013887,-2.3014421)
\psline[linecolor=black, linewidth=0.08](10.601389,-0.70144224)(8.601389,-2.3014421)(8.601389,-2.3014421)
\psline[linecolor=black, linewidth=0.08](10.601389,-0.70144224)(10.601389,-2.3014421)(12.201389,-2.3014421)(11.801389,-2.3014421)
\psline[linecolor=black, linewidth=0.08](12.201389,-2.3014421)(10.601389,-0.70144224)(10.601389,-0.70144224)
\psline[linecolor=black, linewidth=0.08](10.601389,-0.70144224)(13.801389,-2.3014421)(13.801389,-2.3014421)
\psdots[linecolor=black, dotsize=0.4](10.601389,-0.70144224)
\psdots[linecolor=black, dotsize=0.4](12.601389,-0.70144224)
\psdots[linecolor=black, dotsize=0.4](13.801389,-2.3014421)
\psdots[linecolor=black, dotsize=0.4](8.601389,-2.3014421)
\psdots[linecolor=black, fillstyle=solid, dotstyle=o, dotsize=0.4, fillcolor=white](10.601389,-2.3014421)
\psdots[linecolor=black, fillstyle=solid, dotstyle=o, dotsize=0.4, fillcolor=white](12.201389,-2.3014421)
\rput[bl](1.8013889,-5.5014424){\LARGE{$H$}}
\rput[bl](10.601389,-5.5014424){\LARGE{$H-v$}}
\rput[bl](2.5013888,-4.1014423){$v$}
\end{pspicture}
}
		\end{center}
		\caption{Graphs $H$ and $H-v$ with same co-even domination numbers} \label{fig2-example}
	\end{figure}

Now we consider to edge removing of a graph and present upper and lower bound for the constructed graph.

\begin{theorem}\label{G-e}
Let $G=(V,E)$ be a graph and $e\in E$. Then,
$$\gamma _{coe}(G) - 2 \leq \gamma _{coe}(G-e)\leq \gamma _{coe}(G) + 2.$$
\end{theorem}

\begin{proof}
Suppose that $e=uv\in E$ and $D_{coe}(G)$ is co-even dominating set of $G$. First we find the upper bound for $\gamma _{coe}(G-e)$. We consider the following cases:
\begin{itemize}
\item[(i)]
$u,v\notin D_{coe}(G)$. In this case, the degree of these vertices should be even. Now by removing $e$, the degree of these vertices are odd and they should be in co-even dominating set of $G-e$. Now by considering $D_{coe}(G)\cup\{u,v\}$ as a dominating set of $G-e$, we have a co-even dominating set for that with size $\gamma _{coe}(G) + 2$. So $\gamma _{coe}(G-e)\leq \gamma _{coe}(G) + 2$.
\item[(ii)]
$u\in D_{coe}(G)$ and $v\notin D_{coe}(G)$. In this case, the degree of $v$ should be even. Now by removing $e$ and the same argument as previous case, $D_{coe}(G)\cup\{v\}$ is a co-even dominating set of $G-e$ and $\gamma _{coe}(G-e)\leq \gamma _{coe}(G) + 1$.
\item[(iii)]
$u,v\in D_{coe}(G)$. By removing edge $e$ and considering $D_{coe}(G)$ as a domination set of $G-e$, we have a co-even dominating set for $G-e$ too. So $\gamma _{coe}(G-e)\leq \gamma _{coe}(G)$.
\end{itemize}
Now we find the lower bound for $\gamma _{coe}(G-e)$.
First we remove $e$. At this step, we find a co-even dominating set for $G-e$. It is easy to see that $D_{coe}(G-e)\cup \{u,v\}$ is a co-even dominating set of $G$. So 
$$\gamma _{coe}(G)\leq \gamma _{coe}(G-e) +2,$$
and therefore we have the result.\qed
\end{proof}

We end this section by showing that the bounds are sharp in the previous Theorem.

\begin{remark}
The  bounds in Theorem \ref{G-e} are sharp.  For the upper bound, it suffices to consider $G$ as shown in Figure \ref{figg-e}. The set of black vertices in $G$ is a co-even dominating set of $G$. Now, by removing edge $e$, the set of black vertices is a co-even dominating set of $G-e$, and $\gamma _{coe}(G-e)= \gamma _{coe}(G) + 2$. For the lower bound, it suffices to consider $H$ as shown in Figure \ref{figg-elower}. The set of black vertices in $H$ and $H-e$ are co-even dominating sets of them, respectively. So $\gamma _{coe}(H-e)= \gamma _{coe}(H) - 2$.
\end{remark}

	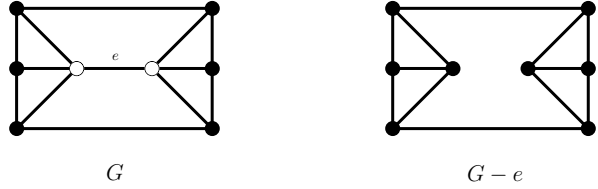
\begin{figure}
		\begin{center}
			\psscalebox{0.5 0.5}
{
\begin{pspicture}(0,-6.6214423)(15.594231,-1.7843268)
\psline[linecolor=black, linewidth=0.08](1.7971154,-3.5814424)(3.7971156,-3.5814424)(3.7971156,-3.5814424)
\psline[linecolor=black, linewidth=0.08](3.7971156,-3.5814424)(5.3971157,-1.9814422)(5.3971157,-3.5814424)(5.3971157,-5.1814423)(3.7971156,-3.5814424)(3.7971156,-3.5814424)
\psline[linecolor=black, linewidth=0.08](3.7971156,-3.5814424)(5.3971157,-3.5814424)(5.3971157,-3.5814424)
\psline[linecolor=black, linewidth=0.08](1.7971154,-3.5814424)(0.19711548,-1.9814422)(0.19711548,-5.1814423)(1.7971154,-3.5814424)(0.19711548,-3.5814424)(0.19711548,-3.5814424)
\psline[linecolor=black, linewidth=0.08](0.19711548,-1.9814422)(5.3971157,-1.9814422)(5.3971157,-1.9814422)
\psline[linecolor=black, linewidth=0.08](0.19711548,-5.1814423)(5.3971157,-5.1814423)(5.3971157,-5.1814423)
\psdots[linecolor=black, dotsize=0.4](0.19711548,-1.9814422)
\psdots[linecolor=black, dotsize=0.4](5.3971157,-1.9814422)
\psdots[linecolor=black, dotsize=0.4](5.3971157,-3.5814424)
\psdots[linecolor=black, dotsize=0.4](5.3971157,-5.1814423)
\psdots[linecolor=black, dotsize=0.4](0.19711548,-5.1814423)
\psdots[linecolor=black, dotsize=0.4](0.19711548,-3.5814424)
\psline[linecolor=black, linewidth=0.08](13.797115,-3.5814424)(15.397116,-1.9814422)(15.397116,-3.5814424)(15.397116,-5.1814423)(13.797115,-3.5814424)(13.797115,-3.5814424)
\psline[linecolor=black, linewidth=0.08](13.797115,-3.5814424)(15.397116,-3.5814424)(15.397116,-3.5814424)
\psline[linecolor=black, linewidth=0.08](11.797115,-3.5814424)(10.197116,-1.9814422)(10.197116,-5.1814423)(11.797115,-3.5814424)(10.197116,-3.5814424)(10.197116,-3.5814424)
\psline[linecolor=black, linewidth=0.08](10.197116,-1.9814422)(15.397116,-1.9814422)(15.397116,-1.9814422)
\psline[linecolor=black, linewidth=0.08](10.197116,-5.1814423)(15.397116,-5.1814423)(15.397116,-5.1814423)
\psdots[linecolor=black, dotsize=0.4](10.197116,-1.9814422)
\psdots[linecolor=black, dotsize=0.4](15.397116,-1.9814422)
\psdots[linecolor=black, dotsize=0.4](15.397116,-3.5814424)
\psdots[linecolor=black, dotsize=0.4](15.397116,-5.1814423)
\psdots[linecolor=black, dotsize=0.4](10.197116,-5.1814423)
\psdots[linecolor=black, dotsize=0.4](10.197116,-3.5814424)
\psdots[linecolor=black, dotsize=0.4](11.797115,-3.5814424)
\psdots[linecolor=black, dotsize=0.4](13.797115,-3.5814424)
\psdots[linecolor=black, dotstyle=o, dotsize=0.4, fillcolor=white](1.7971154,-3.5814424)
\rput[bl](2.5771155,-6.5614424){\LARGE{$G$}}
\rput[bl](12.177115,-6.6214423){\LARGE{$G-e$}}
\rput[bl](2.7371154,-3.3614423){$e$}
\psdots[linecolor=black, dotstyle=o, dotsize=0.4, fillcolor=white](3.7971156,-3.5814424)
\end{pspicture}
}
		\end{center}
		\caption{Graphs $G$ and $G-e$} \label{figg-e}
	\end{figure}

	\begin{figure}
		\begin{center}
			\psscalebox{0.5 0.5}
{
\begin{pspicture}(0,-5.821442)(15.602778,-2.5843267)
\psline[linecolor=black, linewidth=0.08](1.8013889,-4.381442)(3.8013887,-4.381442)(3.8013887,-4.381442)
\psline[linecolor=black, linewidth=0.08](3.8013887,-4.381442)(5.4013886,-4.381442)(5.4013886,-4.381442)
\psline[linecolor=black, linewidth=0.08](0.20138885,-2.7814422)(5.4013886,-2.7814422)(5.4013886,-2.7814422)
\psdots[linecolor=black, dotsize=0.4](0.20138885,-2.7814422)
\psdots[linecolor=black, dotsize=0.4](5.4013886,-2.7814422)
\psline[linecolor=black, linewidth=0.08](13.801389,-4.381442)(15.401389,-4.381442)(15.401389,-4.381442)
\psline[linecolor=black, linewidth=0.08](10.201389,-2.7814422)(15.401389,-2.7814422)(15.401389,-2.7814422)
\psdots[linecolor=black, dotsize=0.4](10.201389,-2.7814422)
\psdots[linecolor=black, dotsize=0.4](15.401389,-2.7814422)
\rput[bl](2.7413888,-4.1614423){$e$}
\psline[linecolor=black, linewidth=0.08](5.4013886,-4.381442)(5.4013886,-2.7814422)(3.8013887,-4.381442)(3.8013887,-4.381442)
\psline[linecolor=black, linewidth=0.08](1.8013889,-4.381442)(0.20138885,-2.7814422)(0.20138885,-4.381442)(1.8013889,-4.381442)(1.8013889,-4.381442)
\psdots[linecolor=black, dotstyle=o, dotsize=0.4, fillcolor=white](5.4013886,-4.381442)
\psdots[linecolor=black, dotstyle=o, dotsize=0.4, fillcolor=white](0.20138885,-4.381442)
\psdots[linecolor=black, dotsize=0.4](1.8013889,-4.381442)
\psdots[linecolor=black, dotsize=0.4](3.8013887,-4.381442)
\psline[linecolor=black, linewidth=0.08](10.201389,-2.7814422)(11.801389,-4.381442)(10.201389,-4.381442)(10.201389,-2.7814422)(10.201389,-2.7814422)
\psline[linecolor=black, linewidth=0.08](15.401389,-2.7814422)(13.801389,-4.381442)(13.801389,-4.381442)
\psline[linecolor=black, linewidth=0.08](15.401389,-4.381442)(15.401389,-2.7814422)(15.401389,-2.7814422)(15.401389,-2.7814422)
\psdots[linecolor=black, dotstyle=o, dotsize=0.4, fillcolor=white](15.401389,-4.381442)
\psdots[linecolor=black, dotstyle=o, dotsize=0.4, fillcolor=white](13.801389,-4.381442)
\psdots[linecolor=black, dotstyle=o, dotsize=0.4, fillcolor=white](11.801389,-4.381442)
\psdots[linecolor=black, dotstyle=o, dotsize=0.4, fillcolor=white](10.201389,-4.381442)
\rput[bl](2.581389,-5.761442){\LARGE{$H$}}
\rput[bl](12.181389,-5.821442){\LARGE{$H-e$}}
\end{pspicture}
}
		\end{center}
		\caption{Graphs $H$ and $H-e$} \label{figg-elower}
	\end{figure}
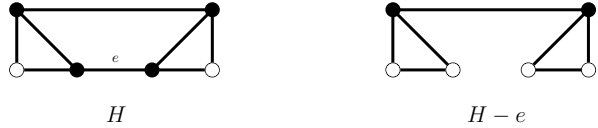

\section{Vertex and edge contraction of a graph}

 Let $v$ be a vertex in graph $G$. The contraction of $v$ in $G$ denoted by $G/v$ is the graph obtained by deleting $v$ and putting a clique on the open neighbourhood of $v$. Note that this operation does not create parallel edges; if two neighbours of $v$ are already adjacent, then they remain simply adjacent (see \cite{Walsh}).  In a graph $G$, contraction of an edge $e$ with endpoints $u,v$ is the replacement of $u$ and $v$ with a single vertex such that edges incident to the new vertex are the edges other than $e$ that were incident with $u$ or $v$. The resulting graph $G/e$ has one less edge than $G$ (\cite{Bondy}). We denote this graph by $G/e$. 
In this section we  examine the effects on $\gamma _{coe}(G)$ when $G$ is modified by an edge contraction and vertex contraction.
 First, we consider to the vertex contraction of a graph and find upper  and lower bound for co-even domination number of that.

\begin{theorem}\label{G/v}
Let $G=(V,E)$ be a graph and $v\in V$. Then,
$$ \gamma _{coe}(G) - \deg(v) -1 \leq \gamma _{coe}(G/v)\leq \gamma _{coe}(G) + \deg(v) -1.$$
\end{theorem}

\begin{proof}
 Suppose that $v\in V$ and $D_{coe}(G)$ is co-even dominating set of $G$. First we find the upper bound for $\gamma _{coe}(G/v)$. We consider the following cases:
\begin{itemize}
\item[(i)]
$\deg(v)$ is odd. So $v\in D_{coe}(G)$. Now, by deleting $v$ and putting a
clique on the open neighbourhood of $v$, we have $G/v$. One can easily check that 
$$\left( D_{coe}(G)-\{v\} \right) \cup N_G(v)$$
 is a co-even dominating set for $G/v$ with size at most $\gamma _{coe}(G)-1 + \deg(v)$. So in this case, $\gamma _{coe}(G/v)\leq \gamma _{coe}(G) + \deg(v) -1$.
\item[(ii)]
$\deg(v)$ is even and $v\in D_{coe}(G)$. By the same argument as previous case, we conclude that $\gamma _{coe}(G/v)\leq \gamma _{coe}(G) + \deg(v) -1$.
\item[(iii)]
$\deg(v)$ is even and $v\notin D_{coe}(G)$. Then at least one vertex in $N_G(v)$ should be in $D_{coe}(G)$. Now $D_{coe}(G)\cup N_G(v)$ is a co-even dominating set for $G/v$ with size at most $\gamma _{coe}(G)-1 + \deg(v)$, and $\gamma _{coe}(G/v)\leq \gamma _{coe}(G) + \deg(v) -1$.
\end{itemize}
Therefore $\gamma _{coe}(G/v)\leq \gamma _{coe}(G) + \deg(v) -1$. Now we find the lower bound for $\gamma _{coe}(G/v)$.
First we remove $v$ and put a clique in open neighbourhood of that. Now we find a co-even dominating set for $G/v$. It is possible that we have $t$ vertices from $N_G(v)$ in $D_{coe}(G/v)$, where $0 \leq t \leq \deg(v)$. Now we keep our dominating set for $G/v$ and remove all the edges we added before and add vertex $v$ and all corresponding edges to that.  In any case, $D_{coe}(G/v) \cup N_G[v]$ is a co-even dominating set for $G$. So  
$$\gamma _{coe}(G)\leq \gamma _{coe}(G/v) + \deg(v) +1,$$
and
therefore we have the result.\qed
\end{proof}

\begin{remark}
The  bounds in Theorem \ref{G/v} are sharp. For the upper bound, it suffices to consider $G$ as shown in Figure \ref{G/vupper}. The set of black vertices in $G$ and $G/v$ are  co-even dominating sets of $G$ and $G/v$, respectively. Hence $\gamma _{coe}(G/v)= \gamma _{coe}(G) + \deg(v) -1$. 
For the lower bound, it suffices to consider $H$ as shown in Figure \ref{G/vlower}. The set of black vertices in $H$ and $H/v$ are co-even dominating sets of them, respectively. So $\gamma _{coe}(H/v)= \gamma _{coe}(H) - \deg(v) -1$.
\end{remark}

	\begin{figure}
		\begin{center}
			\psscalebox{0.5 0.5}
{
\begin{pspicture}(0,-7.5393057)(12.794231,-2.4579167)
\rput[bl](2.1971154,-7.459306){\LARGE{$G$}}
\rput[bl](0.01711548,-4.8593054){$v$}
\rput[bl](10.517116,-7.5393057){\LARGE{$G/v$}}
\psline[linecolor=black, linewidth=0.08](0.19711548,-4.2593055)(1.7971154,-4.2593055)(1.7971154,-4.2593055)
\psline[linecolor=black, linewidth=0.08](0.19711548,-4.2593055)(2.5971155,-2.6593056)(2.5971155,-2.6593056)
\psline[linecolor=black, linewidth=0.08](0.19711548,-4.2593055)(2.5971155,-5.8593054)(2.5971155,-5.8593054)
\psline[linecolor=black, linewidth=0.08](0.19711548,-4.2593055)(4.1971154,-2.6593056)(4.1971154,-2.6593056)
\psline[linecolor=black, linewidth=0.08](0.19711548,-4.2593055)(4.1971154,-5.8593054)(4.1971154,-5.8593054)
\psline[linecolor=black, linewidth=0.08](4.9971156,-4.2593055)(4.1971154,-2.6593056)(4.1971154,-2.6593056)
\psline[linecolor=black, linewidth=0.08](4.9971156,-4.2593055)(2.5971155,-2.6593056)(2.5971155,-2.6593056)
\psline[linecolor=black, linewidth=0.08](4.9971156,-4.2593055)(4.1971154,-5.8593054)(4.1971154,-5.8593054)
\psline[linecolor=black, linewidth=0.08](4.9971156,-4.2593055)(2.5971155,-5.8593054)(2.5971155,-5.8593054)
\psline[linecolor=black, linewidth=0.08](4.9971156,-4.2593055)(1.7971154,-4.2593055)(1.7971154,-4.2593055)
\psdots[linecolor=black, dotstyle=o, dotsize=0.4, fillcolor=white](1.7971154,-4.2593055)
\psdots[linecolor=black, dotstyle=o, dotsize=0.4, fillcolor=white](2.5971155,-5.8593054)
\psdots[linecolor=black, dotstyle=o, dotsize=0.4, fillcolor=white](4.1971154,-5.8593054)
\psdots[linecolor=black, dotstyle=o, dotsize=0.4, fillcolor=white](4.1971154,-2.6593056)
\psdots[linecolor=black, dotstyle=o, dotsize=0.4, fillcolor=white](2.5971155,-2.6593056)
\psdots[linecolor=black, dotsize=0.4](4.9971156,-4.2593055)
\psdots[linecolor=black, dotsize=0.4](0.19711548,-4.2593055)
\psline[linecolor=black, linewidth=0.08](12.5971155,-4.2593055)(11.797115,-2.6593056)(11.797115,-2.6593056)
\psline[linecolor=black, linewidth=0.08](12.5971155,-4.2593055)(10.197116,-2.6593056)(10.197116,-2.6593056)
\psline[linecolor=black, linewidth=0.08](12.5971155,-4.2593055)(11.797115,-5.8593054)(11.797115,-5.8593054)
\psline[linecolor=black, linewidth=0.08](12.5971155,-4.2593055)(10.197116,-5.8593054)(10.197116,-5.8593054)
\psline[linecolor=black, linewidth=0.08](12.5971155,-4.2593055)(9.397116,-4.2593055)(9.397116,-4.2593055)
\psdots[linecolor=black, dotsize=0.4](12.5971155,-4.2593055)
\psline[linecolor=black, linewidth=0.08](11.797115,-2.6593056)(10.197116,-2.6593056)(9.397116,-4.2593055)(10.197116,-5.8593054)(11.797115,-5.8593054)(11.797115,-2.6593056)(9.397116,-4.2593055)(11.797115,-5.8593054)(10.197116,-2.6593056)(10.197116,-5.8593054)(10.197116,-5.8593054)
\psline[linecolor=black, linewidth=0.08](11.797115,-2.6593056)(10.197116,-5.8593054)(10.197116,-5.459306)
\psdots[linecolor=black, dotsize=0.4](11.797115,-2.6593056)
\psdots[linecolor=black, dotsize=0.4](10.197116,-2.6593056)
\psdots[linecolor=black, dotsize=0.4](9.397116,-4.2593055)
\psdots[linecolor=black, dotsize=0.4](10.197116,-5.8593054)
\psdots[linecolor=black, dotsize=0.4](11.797115,-5.8593054)
\end{pspicture}
}
		\end{center}
		\caption{Graphs $G$ and $G/v$} \label{G/vupper}
	\end{figure}
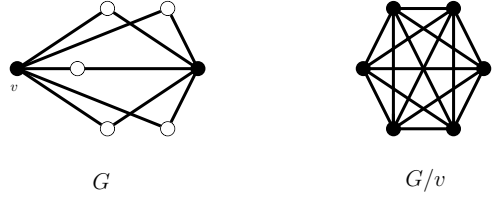

	\begin{figure}
		\begin{center}
			\psscalebox{0.5 0.5}
{
\begin{pspicture}(0,-5.7514424)(15.994231,1.3456731)
\rput[bl](2.9171152,-5.651442){\LARGE{$H$}}
\rput[bl](11.957115,-5.7514424){\LARGE{$H/v$}}
\psdots[linecolor=black, dotsize=0.4](3.3971152,-1.6514423)
\psdots[linecolor=black, dotsize=0.4](4.997115,-0.45144227)
\psdots[linecolor=black, dotsize=0.4](3.3971152,-3.2514422)
\psdots[linecolor=black, dotsize=0.4](1.7971153,-0.45144227)
\psdots[linecolor=black, dotsize=0.4](4.997115,1.1485578)
\psdots[linecolor=black, dotsize=0.4](6.5971155,-0.45144227)
\psdots[linecolor=black, dotsize=0.4](4.5971155,-4.4514422)
\psdots[linecolor=black, dotsize=0.4](2.1971154,-4.4514422)
\psdots[linecolor=black, dotsize=0.4](1.7971153,1.1485578)
\psdots[linecolor=black, dotsize=0.4](0.19711533,-0.45144227)
\psline[linecolor=black, linewidth=0.08](0.19711533,-0.45144227)(1.7971153,-0.45144227)
\psline[linecolor=black, linewidth=0.08](1.7971153,1.1485578)(1.7971153,-0.45144227)
\psline[linecolor=black, linewidth=0.08](1.7971153,-0.45144227)(3.3971152,-1.6514423)
\psline[linecolor=black, linewidth=0.08](3.3971152,-1.6514423)(4.997115,-0.45144227)
\psline[linecolor=black, linewidth=0.08](4.997115,-0.45144227)(4.997115,1.1485578)
\psline[linecolor=black, linewidth=0.08](4.997115,-0.45144227)(6.5971155,-0.45144227)
\psline[linecolor=black, linewidth=0.08](2.1971154,-4.4514422)(3.3971152,-3.2514422)(3.3971152,-3.2514422)
\psline[linecolor=black, linewidth=0.08](3.3971152,-3.2514422)(4.5971155,-4.4514422)(4.5971155,-4.4514422)
\rput[bl](2.9971154,-2.0514421){$v$}
\psdots[linecolor=black, dotsize=0.4](14.197115,1.1485578)
\psdots[linecolor=black, dotsize=0.4](15.797115,-0.45144227)
\psdots[linecolor=black, dotsize=0.4](13.797115,-4.4514422)
\psdots[linecolor=black, dotsize=0.4](11.397116,-4.4514422)
\psdots[linecolor=black, dotsize=0.4](10.997115,1.1485578)
\psdots[linecolor=black, dotsize=0.4](9.397116,-0.45144227)
\psline[linecolor=black, linewidth=0.08](9.397116,-0.45144227)(10.997115,-0.45144227)
\psline[linecolor=black, linewidth=0.08](10.997115,1.1485578)(10.997115,-0.45144227)
\psline[linecolor=black, linewidth=0.08](14.197115,-0.45144227)(14.197115,1.1485578)
\psline[linecolor=black, linewidth=0.08](14.197115,-0.45144227)(15.797115,-0.45144227)
\psline[linecolor=black, linewidth=0.08](11.397116,-4.4514422)(12.5971155,-3.2514422)(12.5971155,-3.2514422)
\psline[linecolor=black, linewidth=0.08](12.5971155,-3.2514422)(13.797115,-4.4514422)(13.797115,-4.4514422)
\psline[linecolor=black, linewidth=0.08](10.997115,-0.45144227)(14.197115,-0.45144227)
\psline[linecolor=black, linewidth=0.08](14.197115,-0.45144227)(12.5971155,-3.2514422)
\psline[linecolor=black, linewidth=0.08](10.997115,-0.45144227)(12.5971155,-3.2514422)
\psline[linecolor=black, linewidth=0.08](3.3971152,-1.6514423)(3.3971152,-3.2514422)(3.3971152,-3.2514422)
\psdots[linecolor=black, dotstyle=o, dotsize=0.4, fillcolor=white](10.997115,-0.45144227)
\psdots[linecolor=black, dotstyle=o, dotsize=0.4, fillcolor=white](14.197115,-0.45144227)
\psdots[linecolor=black, dotstyle=o, dotsize=0.4, fillcolor=white](12.5971155,-3.2514422)
\end{pspicture}
}
		\end{center}
		\caption{Graphs $H$ and $H/v$} \label{G/vlower}
	\end{figure}
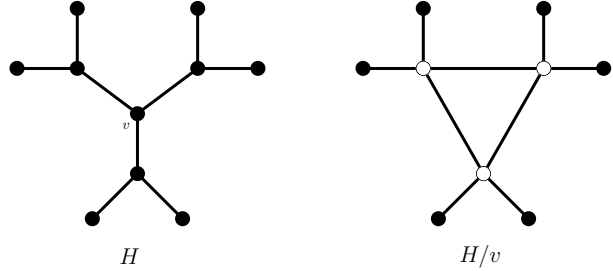

As an immediate result of Theorems \ref{G-v} and \ref{G/v}, we have:

\begin{corollary}
Let $G=(V,E)$ be a graph and $v\in V$. Then,
$$ \frac{ \gamma _{coe}(G-v) + \gamma _{coe}(G/v) }{2} - \deg(v) +1 \leq \gamma _{coe}(G)\leq \frac{ \gamma _{coe}(G-v) + \gamma _{coe}(G/v) }{2} + \deg(v) +1.$$
\end{corollary}

Now we consider to the edge contraction of a graph and find upper  and lower bound for co-even domination number of that.

\begin{theorem}\label{G/e}
Let $G=(V,E)$ be a graph and $e\in E$. Then,
$$\gamma _{coe}(G) - 2 \leq \gamma _{coe}(G/e)\leq \gamma _{coe}(G).$$
\end{theorem}

\begin{proof}
Suppose that $e=uv\in E$ and $D_{coe}(G)$ is co-even dominating set of $G$. Also let $w$ be the vertex which is replacement of $u$ and $v$ in $G/e$. First, we find the upper bound for $\gamma _{coe}(G/e)$. We consider the following cases:
\begin{itemize}
\item[(i)]
$u,v\notin D_{coe}(G)$. In this case, the degree of these vertices should be even. Now by removing $e$, the degree of these vertices are odd and therefore the degree of $w$ is even in $G/e$. Now by considering $D_{coe}(G)$ as a dominating set of $G/e$ too, we have a co-even dominating set for that with size $\gamma _{coe}(G)$. So $\gamma _{coe}(G/e)\leq \gamma _{coe}(G)$.
\item[(ii)]
$u\in D_{coe}(G)$ and $v\notin D_{coe}(G)$. In this case, $\left( D_{coe}(G)-\{u\} \right) \cup\{w\}$ is a co-even dominating set of $G/e$ and $\gamma _{coe}(G/e)\leq \gamma _{coe}(G)$.
\item[(iii)]
$u,v\in D_{coe}(G)$. In this case, $\left( D_{coe}(G)-\{u,v\} \right) \cup\{w\}$ is a co-even dominating set of $G/e$ and $\gamma _{coe}(G/e)\leq \gamma _{coe}(G)-1$.
\end{itemize}
So $\gamma _{coe}(G/e)\leq \gamma _{coe}(G)$. Now we find the lower bound for $\gamma _{coe}(G/e)$.
First we consider to $G/e$ and find a co-even dominating set for that. In the worst case, $w\notin D_{coe}(G/e)$ and the degree of $u$ and $v$ are odd in $G$. 
 So in any case, $\left( D_{coe}(G)-\{w\} \right) \cup\{u,v\}$  is a co-even dominating set for $G$. Hence
$\gamma _{coe}(G)\leq \gamma _{coe}(G/e) +2$,
and therefore we have the result.\qed
\end{proof}

\begin{remark}\label{remark}
The  bounds in Theorem \ref{G/e} are sharp.  For the upper bound, it suffices to consider $G$ as shown in Figure \ref{G/eupper}. The set of black vertices in $G$ is a co-even dominating set of $G$. Also, the set of black vertices in $G/e$ is a co-even dominating set of $G/e$, and $\gamma _{coe}(G/e)= \gamma _{coe}(G) $. For the lower bound, it suffices to consider $H$ as shown in Figure \ref{G/elower}. The set of black vertices in $H$ and $H/e$ are co-even dominating sets of them, respectively. So $\gamma _{coe}(H/e)= \gamma _{coe}(H) - 2$.
\end{remark}

	\begin{figure}
		\begin{center}
			\psscalebox{0.5 0.5}
{
\begin{pspicture}(0,-7.5414424)(13.194231,-4.064327)
\rput[bl](2.5971153,-7.4614425){\LARGE{$G$}}
\rput[bl](4.0771155,-5.9414425){$v$}
\rput[bl](10.517116,-7.5414424){\LARGE{$G/e$}}
\rput[bl](2.0371153,-5.9814425){$u$}
\rput[bl](3.0571153,-5.241442){$e$}
\psdots[linecolor=black, dotsize=0.4](0.9971153,-6.6614423)
\psdots[linecolor=black, dotsize=0.4](0.9971153,-4.261442)
\psline[linecolor=black, linewidth=0.08](4.1971154,-5.4614425)(2.1971154,-5.4614425)(2.1971154,-5.4614425)
\psdots[linecolor=black, dotsize=0.4](0.19711533,-5.4614425)
\psline[linecolor=black, linewidth=0.08](2.1971154,-5.4614425)(0.9971153,-4.261442)(0.9971153,-4.261442)
\psline[linecolor=black, linewidth=0.08](2.1971154,-5.4614425)(0.9971153,-6.6614423)(0.9971153,-6.6614423)
\psline[linecolor=black, linewidth=0.08](2.1971154,-5.4614425)(0.19711533,-5.4614425)(0.19711533,-5.4614425)
\psdots[linecolor=black, dotsize=0.4](5.397115,-4.261442)
\psdots[linecolor=black, dotsize=0.4](6.1971154,-5.4614425)
\psdots[linecolor=black, dotsize=0.4](5.397115,-6.6614423)
\psline[linecolor=black, linewidth=0.08](4.1971154,-5.4614425)(5.397115,-4.261442)(5.397115,-4.261442)
\psline[linecolor=black, linewidth=0.08](4.1971154,-5.4614425)(6.1971154,-5.4614425)
\psline[linecolor=black, linewidth=0.08](4.1971154,-5.4614425)(5.397115,-6.6614423)
\psdots[linecolor=black, dotstyle=o, dotsize=0.4, fillcolor=white](2.1971154,-5.4614425)
\psdots[linecolor=black, dotstyle=o, dotsize=0.4, fillcolor=white](4.1971154,-5.4614425)
\psdots[linecolor=black, dotsize=0.4](9.797115,-6.6614423)
\psdots[linecolor=black, dotsize=0.4](9.797115,-4.261442)
\psdots[linecolor=black, dotsize=0.4](8.997115,-5.4614425)
\psline[linecolor=black, linewidth=0.08](10.997115,-5.4614425)(9.797115,-4.261442)(9.797115,-4.261442)
\psline[linecolor=black, linewidth=0.08](10.997115,-5.4614425)(9.797115,-6.6614423)(9.797115,-6.6614423)
\psline[linecolor=black, linewidth=0.08](10.997115,-5.4614425)(8.997115,-5.4614425)(8.997115,-5.4614425)
\psdots[linecolor=black, dotsize=0.4](12.197115,-4.261442)
\psdots[linecolor=black, dotsize=0.4](12.997115,-5.4614425)
\psdots[linecolor=black, dotsize=0.4](12.197115,-6.6614423)
\psline[linecolor=black, linewidth=0.08](10.997115,-5.4614425)(12.197115,-4.261442)(12.197115,-4.261442)
\psline[linecolor=black, linewidth=0.08](10.997115,-5.4614425)(12.997115,-5.4614425)
\psline[linecolor=black, linewidth=0.08](10.997115,-5.4614425)(12.197115,-6.6614423)
\psdots[linecolor=black, dotstyle=o, dotsize=0.4, fillcolor=white](10.997115,-5.4614425)
\psdots[linecolor=black, dotstyle=o, dotsize=0.4, fillcolor=white](10.997115,-5.4614425)
\rput[bl](10.837115,-5.0614424){$w$}
\end{pspicture}
}
		\end{center}
		\caption{Graphs $G$ and $G/e$} \label{G/eupper}
	\end{figure}

	\begin{figure}
		\begin{center}
			\psscalebox{0.5 0.5}
{
\begin{pspicture}(0,-7.5414424)(11.594231,-4.064327)
\rput[bl](3.2771156,-5.9414425){$v$}
\rput[bl](1.2371155,-5.9814425){$u$}
\rput[bl](2.2571154,-5.241442){$e$}
\psdots[linecolor=black, dotsize=0.4](0.19711548,-6.6614423)
\psdots[linecolor=black, dotsize=0.4](0.19711548,-4.261442)
\psline[linecolor=black, linewidth=0.08](3.3971155,-5.4614425)(1.3971155,-5.4614425)(1.3971155,-5.4614425)
\psline[linecolor=black, linewidth=0.08](1.3971155,-5.4614425)(0.19711548,-4.261442)(0.19711548,-4.261442)
\psline[linecolor=black, linewidth=0.08](1.3971155,-5.4614425)(0.19711548,-6.6614423)(0.19711548,-6.6614423)
\psdots[linecolor=black, dotsize=0.4](4.5971155,-4.261442)
\psdots[linecolor=black, dotsize=0.4](4.5971155,-6.6614423)
\psline[linecolor=black, linewidth=0.08](3.3971155,-5.4614425)(4.5971155,-4.261442)(4.5971155,-4.261442)
\psline[linecolor=black, linewidth=0.08](3.3971155,-5.4614425)(4.5971155,-6.6614423)
\psdots[linecolor=black, dotsize=0.4](8.997115,-6.6614423)
\psdots[linecolor=black, dotsize=0.4](8.997115,-4.261442)
\psline[linecolor=black, linewidth=0.08](10.197116,-5.4614425)(8.997115,-4.261442)(8.997115,-4.261442)
\psline[linecolor=black, linewidth=0.08](10.197116,-5.4614425)(8.997115,-6.6614423)(8.997115,-6.6614423)
\psdots[linecolor=black, dotsize=0.4](11.397116,-4.261442)
\psdots[linecolor=black, dotsize=0.4](11.397116,-6.6614423)
\psline[linecolor=black, linewidth=0.08](10.197116,-5.4614425)(11.397116,-4.261442)(11.397116,-4.261442)
\psline[linecolor=black, linewidth=0.08](10.197116,-5.4614425)(11.397116,-6.6614423)
\psdots[linecolor=black, dotstyle=o, dotsize=0.4, fillcolor=white](10.197116,-5.4614425)
\psdots[linecolor=black, dotstyle=o, dotsize=0.4, fillcolor=white](10.197116,-5.4614425)
\rput[bl](10.037115,-5.0614424){$w$}
\rput[bl](1.7971154,-7.4614425){\LARGE{$H$}}
\rput[bl](9.717115,-7.5414424){\LARGE{$H/e$}}
\psdots[linecolor=black, dotsize=0.4](3.3971155,-5.4614425)
\psdots[linecolor=black, dotsize=0.4](1.3971155,-5.4614425)
\end{pspicture}
}
		\end{center}
		\caption{Graphs $H$ and $H/e$} \label{G/elower}
	\end{figure}
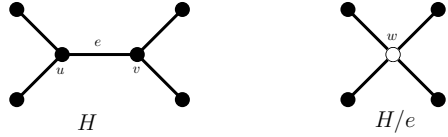

In Theorem \ref{G/e}, we showed that
$\gamma _{coe}(G) - 2 \leq \gamma _{coe}(G/e)\leq \gamma _{coe}(G),$ for every $e\in E$. Also in Remark \ref{remark}, we  concluded that these bounds are sharp. In 
the following example, we show that $\gamma _{coe}(G/e)$ can be  $\gamma _{coe}(G)-1$ too:

\begin{example}
Consider $G$ as shown in Figure \ref{G/eexm}. The set of black vertices in $G$ is a co-even dominating set of $G$. Also, the set of black vertices in $G/e$ is a co-even dominating set of $G/e$, and $\gamma _{coe}(G/e)= \gamma _{coe}(G) - 1$.
\end{example}

	\begin{figure}[!h]
		\begin{center}
			\psscalebox{0.5 0.5}
{
\begin{pspicture}(0,-7.341442)(11.998505,-3.4643269)
\rput[bl](1.7971154,-7.261442){\LARGE{$G$}}
\rput[bl](3.2771156,-5.741442){$v$}
\rput[bl](9.717115,-7.341442){\LARGE{$G/e$}}
\rput[bl](1.2371155,-5.781442){$u$}
\rput[bl](10.057116,-5.761442){$w$}
\rput[bl](2.2571154,-5.0414424){$e$}
\psdots[linecolor=black, dotsize=0.4](0.19711548,-6.0614424)
\psdots[linecolor=black, dotsize=0.4](0.19711548,-4.4614425)
\psdots[linecolor=black, dotsize=0.4](1.3971155,-5.261442)
\psdots[linecolor=black, dotsize=0.4](3.3971155,-5.261442)
\psdots[linecolor=black, dotsize=0.4](3.3971155,-3.6614423)
\psline[linecolor=black, linewidth=0.08](3.3971155,-5.261442)(1.3971155,-5.261442)(1.3971155,-5.261442)
\psline[linecolor=black, linewidth=0.08](1.3971155,-5.261442)(0.19711548,-4.4614425)(0.19711548,-4.4614425)
\psline[linecolor=black, linewidth=0.08](1.3971155,-5.261442)(0.19711548,-6.0614424)(0.19711548,-6.0614424)
\psline[linecolor=black, linewidth=0.08](3.3971155,-5.261442)(4.9971156,-4.4614425)(4.9971156,-6.0614424)(3.3971155,-5.261442)(3.3971155,-5.261442)
\psdots[linecolor=black, dotstyle=o, dotsize=0.4, fillcolor=white](4.9971156,-4.4614425)
\psdots[linecolor=black, dotstyle=o, dotsize=0.4, fillcolor=white](4.9971156,-6.0614424)
\psline[linecolor=black, linewidth=0.08](3.3971155,-3.6614423)(3.3971155,-5.261442)(3.3971155,-5.261442)
\psdots[linecolor=black, dotsize=0.4](8.997115,-6.0614424)
\psdots[linecolor=black, dotsize=0.4](8.997115,-4.4614425)
\psdots[linecolor=black, dotsize=0.4](10.197116,-5.261442)
\psdots[linecolor=black, dotsize=0.4](10.197116,-5.261442)
\psdots[linecolor=black, dotsize=0.4](10.197116,-3.6614423)
\psline[linecolor=black, linewidth=0.08](10.197116,-5.261442)(8.997115,-4.4614425)(8.997115,-4.4614425)
\psline[linecolor=black, linewidth=0.08](10.197116,-5.261442)(8.997115,-6.0614424)(8.997115,-6.0614424)
\psline[linecolor=black, linewidth=0.08](10.197116,-5.261442)(11.797115,-4.4614425)(11.797115,-6.0614424)(10.197116,-5.261442)(10.197116,-5.261442)
\psdots[linecolor=black, dotstyle=o, dotsize=0.4, fillcolor=white](11.797115,-4.4614425)
\psdots[linecolor=black, dotstyle=o, dotsize=0.4, fillcolor=white](11.797115,-6.0614424)
\psline[linecolor=black, linewidth=0.08](10.197116,-3.6614423)(10.197116,-5.261442)(10.197116,-5.261442)
\end{pspicture}
}
		\end{center}
		\caption{Graphs $G$ and $G/e$} \label{G/eexm}
	\end{figure}

We end this section by an immediate result of Theorems \ref{G-e} and \ref{G/e}:

\begin{corollary}
Let $G=(V,E)$ be a graph and $e\in E$. Then,
$$ \frac{ \gamma _{coe}(G-e) + \gamma _{coe}(G/e) }{2} - 1 \leq \gamma _{coe}(G)\leq \frac{ \gamma _{coe}(G-e) + \gamma _{coe}(G/e) }{2} +2.$$
\end{corollary}

\section{Conclusions}

In this paper, we obtained some lower and upper bounds of co-even domination number of graphs which constructed  by vertex and edge removing, and also vertex and edge contraction regarding the co-even domination number of the main graph.  Also
we showed that these bounds are sharp. Then, we presented upper and lower bounds for co-even domination number of a graph regarding vertex (edge) removal and contraction of that as immediate result of our previous results. Future topics of interest for future research include the following suggestions:

\begin{itemize}
\item[(i)]
Finding co-even domination number of other unary operations of graphs such as subdivision of a graph (see \cite{Babu,Nima1} for some results in this topic), power of a graph, etc.
\item[(ii)]
Finding co-even domination number of other operations of graphs such as graph rewriting, line graph, Mycielskian, etc.
\item[(iii)]
Finding co-even domination number of interval graphs, intersection graphs, word-representable graphs, etc.
\end{itemize}


\begin{thebibliography}{99}

	\bibitem{Babu} Ch. S. Babu, A. A. Diwan, Subdivisions of graphs: A generalization of paths and cycles, {\it Discrete  Math.\/}, 308 (2008) 4479–4486.

   \bibitem{Bondy} J.A. Bondy and U.S.R. Murty,  Graph Theory with Applications, {\it American Elsevier, MacMillan, New York, London}, (1976).




	\bibitem{Nima} N. Ghanbari, Co-even domination number of some binary operations on graphs, Submitted. Preprint available at \texttt{https://arxiv.org/abs/2111.01845.}

	\bibitem{Nima1}	N. Ghanbari, Secure domination number of $k$-subdivision of graphs, submitted. Preprint available at \texttt{https://arxiv.org/abs/2110.09190.}

	\bibitem{domination} T.W. Haynes, S.T. Hedetniemi, P.J. Slater, Fundamentals of domination in graphs, {\it Marcel Dekker\/}, NewYork, (1998).




	\bibitem{Sha} M. M. Shalaan, A. A. Omran,   Co-even Domination in Graphs, {\it Int. J. Control Automat.\/}, 13(3) (2020)  330-334.



	 \bibitem{Walsh}	M. Walsh,  The hub number of a graph, {\it Int. J. Math. Comput. Sci.}, 1 (2006) 117-124.
	 
	 
\end{thebibliography}
\end{document}